\documentclass[leqno]{amsart}

\numberwithin{equation}{section} 
\usepackage{amsmath,amsfonts,amssymb,amsthm,latexsym}
\usepackage{enumerate}
\usepackage{cancel}

\usepackage{cases}
\usepackage[square,comma,numbers,sort&compress]{natbib}

\usepackage[colorlinks=true]{hyperref}
\hypersetup{urlcolor=blue, citecolor=red}

\newcounter{mnote}
\theoremstyle{plain}
\newtheorem{theorem}{Theorem}[section]

\newtheorem{lemma}[theorem]{Lemma}

\theoremstyle{definition}
\newtheorem{definition}[theorem]{Definition}

\theoremstyle{remark}
\newtheorem{remark}[theorem]{Remark}

\newcommand{\vect}[1]{\mathbf{#1}}

\newcommand{\bu}{\vect{u}}
\newcommand{\bv}{\vect{v}}
\newcommand{\bw}{\vect{w}}

\newcommand{\bx}{\vect{x}}
\newcommand{\by}{\vect{y}}

\newcommand{\field}[1]{\mathbb{#1}}
\newcommand{\nN}{\field{N}}

\newcommand{\nR}{\field{R}}

\newcommand{\vphi}{\varphi}
\newcommand{\maps}{\rightarrow}

\newcommand{\sand}{\quad\text{and}\quad}

\newcommand{\pd}[2]{\frac{\partial #1}{\partial #2}}

\newcommand{\abs}[1]{\left\lvert#1\right\rvert}
\newcommand{\norm}[1]{\left\lVert#1\right\rVert}
\newcommand{\set}[1]{\left\{#1\right\}}
\newcommand{\ip}[2]{\left<#1,#2\right>}
\newcommand{\pnt}[1]{\left(#1\right)}
\newcommand{\pair}[2]{\left(#1,#2\right)}

\newcommand{\B}{\mathcal{B}}
\newcommand{\byd}{\dot{\mathbf{y}}}
\newcommand{\bg}{\mathbf{g}}
\newcommand{\nT}{\mathbb{T}}

\begin{document}
\title[Voigt-Regularized Inviscid Resistive MHD]{Higher-Order Global Regularity of an Inviscid Voigt-Regularization of the Three-Dimensional Inviscid Resistive Magnetohydrodynamic Equations}

\date{April 3, 2011}

\author{Adam Larios}
\address[Adam Larios]{Department of Mathematics\\
                University of California, Irvine\\
		Irvine CA 92697-3875, USA}
\email[Adam Larios]{alarios@math.uci.edu}
\author{Edriss S. Titi}
\address[Edriss S. Titi]{Department of Mathematics, and Department of Mechanical and Aero-space Engineering, University of California, Irvine, Irvine CA 92697-3875, USA. Also The Department of Computer Science and Applied Mathematics, The Weizmann Institute of Science, Rehovot 76100, Israel. Fellow of the Center of Smart Interfaces, Technische Universit\"at Darmstadt, Germany.}
\email[Edriss S. Titi]{etiti@math.uci.edu and edriss.titi@weizmann.ac.il}


\begin{abstract}
  We prove existence, uniqueness, and higher-order global regularity of strong solutions to a particular Voigt-regularization of the three-dimensional inviscid resistive Magnetohydrodynamic (MHD) equations.  Specifically, the coupling of a resistive magnetic field to the Euler-Voigt model is introduced to form an inviscid regularization of the inviscid resistive MHD system.  The results hold in both the whole space $\nR^3$ and in the context of periodic boundary conditions.   Weak solutions for this regularized model are also considered, and proven to exist globally in time, but the question of uniqueness for weak solutions is still open.  Since the main purpose of this line of research is to introduce a reliable and stable inviscid numerical regularization of the underlying model we, in particular, show that the solutions of the Voigt regularized system converge, as the regularization parameter $\alpha\maps0$, to strong solutions of the original inviscid resistive MHD, on the corresponding time interval of existence of the latter.  Moreover, we also establish a new criterion for blow-up of solutions to the original MHD system inspired by this Voigt regularization.  This type of regularization, and the corresponding results, are valid for, and can also be applied to, a wide class of hydrodynamic models.
  \end{abstract}

 \maketitle
 \thispagestyle{empty}

\section{The Inviscid Resistive MHD-Voigt Model}\label{sec:Intro}
The magnetohydrodynamic equations (MHD) are given by
\begin{subequations}\label{MHD}
\begin{align}
 \partial_t\bu +(\bu\cdot\nabla)\bu
+\nabla (p+\frac{1}{2}|\B|^2)&=(\B\cdot\nabla)\B+\nu\triangle\bu,
\label{MHD_mom}\\
 \partial_t\B +(\bu\cdot\nabla)\B-(\B\cdot\nabla)\bu+\nabla q&=\mu\triangle \B,
\label{MHD_mag}\\
 \nabla\cdot \B = \nabla\cdot \bu &= 0,\label{MHD_div}
\end{align}
\end{subequations}
with appropriate boundary and initial conditions, discussed below.  Here, $\nu\geq0$ is the fluid viscosity, $\mu\geq0$ is the magnetic resistivity, and the unknowns are the fluid velocity field $\bu(\bx,t)=(u_1,u_2,u_3)$, the fluid pressure $p(\bx,t)$, the magnetic field $\B(\bx,t)=(\B_1,\B_2,\B_3)$, and the magnetic pressure $q(\bx,t)$, where $\bx = (x_1,x_2,x_3)$, and $t\geq0$.   Note that, \textit{a posteriori}, one can derive that $\nabla q\equiv 0$.  Due to the fact that these equations contain the three-dimensional Navier-Stokes equations for incompressible flows as a special case (namely, when $\B\equiv0$), the mathematical theory is far from complete.  For a derivation and physical discussion of the MHD equations, see, e.g., \cite{Chandrasekhar_1961}. For an overview of the classical and recent mathematical results pertaining to the MHD equations, see, e.g., \cite{Duvaut_Lions_1972,Davidson_2001}.  In this paper we study the inviscid ($\nu=0$) case, with the following inviscid regularization of \eqref{MHD},
\begin{subequations}\label{MHD_V}
\begin{align}
 -\alpha^2\partial_t \triangle \bu+\partial_t\bu +(\bu\cdot\nabla)\bu
+\nabla (p+\frac{1}{2}|\B|^2)&=(\B\cdot\nabla)\B,
\label{MHD_V_mom}\\
 \partial_t\B +(\bu\cdot\nabla)\B-(\B\cdot\nabla)\bu+\nabla q&=\mu\triangle \B,
\label{MHD_V_mag}\\
 \nabla\cdot \B = \nabla\cdot \bu &= 0,\label{MHD_V_div}\\
 (\bu,\B)|_{t=0}&=(\bu_0,\B_0),\label{MHD_V_in}
\end{align}
\end{subequations}
where $\alpha>0$ is a regularization parameter having units of length, and $\mu>0$.  Note that when $\alpha=0$, we formally retrieve the inviscid ($\nu=0$) resistive system \eqref{MHD}.  Furthermore, adding a forcing term to \eqref{MHD_mom} or \eqref{MHD_mag}, or reintroducing a viscous term $\nu\triangle \bu$ to the right-hand side of  \eqref{MHD_V_mom} (with $\nu>0$), does not pose any additional mathematical difficulties to the results or the analysis of the present work, so it will not be discussed further.

In \cite{Larios_Titi_2009}, we studied inviscid, irresistive ($\nu=0$, $\mu=0$) magnetohydrodynamic equations with an additional inviscid regularization on both the velocity and the magnetic terms.  That is, we studied the system
\begin{subequations}\label{II_MHD_V}
\begin{align}
 -\alpha^2\partial_t \triangle \bu+\partial_t\bu +(\bu\cdot\nabla)\bu
+\nabla (p+\frac{1}{2}|\B|^2)&=(\B\cdot\nabla)\B,
\\
-\alpha_M^2\partial_t \triangle \B+
 \partial_t\B +(\bu\cdot\nabla)\B-(\B\cdot\nabla)\bu+\nabla q&=0,
 \\
 \nabla\cdot \B = \nabla\cdot \bu &= 0,
\end{align}
\end{subequations}
where $\alpha,\alpha_M>0$, the boundary conditions were taken to be periodic, and we also required
\begin{equation}\label{mean_zero}
  \int_{\Omega} \bu\,d\bx =  \int_{\Omega} \B\,d\bx =0.
\end{equation}
Following the ideas of treating the Voigt-regularization of the 3D Euler equations, i.e., the inviscid simplified Bardina model,  presented in \cite{Cao_Lunasin_Titi_2006}, we proved in \cite{Larios_Titi_2009} that if $\bu_0,\B_0\in H^1(\Omega)$ and are divergence free, then \eqref{II_MHD_V} has a unique solution  $(\bu,\B)$ lying in  $C^1((-\infty,\infty),H^1(\Omega))$.  It should be noted that a related regularization known as the simplified Bardina model for the MHD equations has been studied in \cite{Labovsky_Trenchea_2010} in the viscous, resistive ($\nu>0$, $\mu>0$) case, with periodic boundary conditions. This model coincides formally with \eqref{II_MHD_V} in the inviscid, irresistive ($\nu=0$, $\mu=0$) case, a phenomenon which was first explored in \cite{Cao_Lunasin_Titi_2006} in the context of the Euler equations.  We also note that the viscous, irresistive ($\nu>0$, $\mu=0$) Bardina model for the MHD equations in $\nR^2$ was studied in \cite{Catania_2009}.

The simplified Bardina model first arose in the context of turbulence models for the Navier-Stokes equations in \cite{Layton_Lewandowski_2006}.  Based on this work, the authors of \cite{Cao_Lunasin_Titi_2006} studied the simplified Bardina model in the special case where the filtering is given by inverting the Helmholtz operator $I-\alpha^2\triangle$.  In \cite{Cao_Lunasin_Titi_2006}, the authors also studied the inviscid simplified Bardina model with this choice of filtering, which is known as the Euler-Voigt (or sometimes Euler-Voight) model, and proved the in their paper global regularity of solutions to the Euler-Voigt (i.e., inviscid simplified Bardina) system.  Higher-order regularity, including spatial analytic regularity, of solutions to the Euler-Voigt model was then established in \cite{Larios_Titi_2009}.

The particular regularization considered in the present work, known as a Voigt-regularization, belongs to the class of models known as the $\alpha$-models, which have a rich recent history (see, e.g., \cite{Cao_Lunasin_Titi_2006, Foias_Holm_Titi_2002, Ilyin_Lunasin_Titi_2006, Chen_Foias_Holm_Olson_Titi_Wynne_1998_PF, Chen_Foias_Holm_Olson_Titi_Wynne_1999, Holm_Titi_2005, Chen_Foias_Holm_Olson_Titi_Wynne_1998_PRL, Cheskidov_Holm_Olson_Titi_2005}, also see, e.g.,  \cite{Cao_Lunasin_Titi_2006,Larios_Titi_2009} for historical discussions). Voigt-type regularizations in the context of various hydrodynamic models have been the focus of much recent research, see, e.g., \cite{Catania_2009,Catania_Secchi_2009,Cao_Lunasin_Titi_2006,Larios_Titi_2009,Larios_Lunasin_Titi_2010,Ebrahimi_Holst_Lunasin_2010,Levant_Ramos_Titi_2009,Khouider_Titi_2008,Ramos_Titi_2010,Kalantarov_Levant_Titi_2009,Kalantarov_Titi_2009,Khouider_Titi_2008}.  See also \cite{Ebrahimi_Holst_Lunasin_2010} for the application of Navier-Stokes-Voigt model in image inpainting. Voigt-regularizations of parabolic equations are a special case of pseudoparabolic equations, that is, equations of the form $Mu_t+Nu=f$, where $M$ and $N$ are operators which could be non-linear, or even non-local.  For more about pseudoparabolic equations, see, e.g., \cite{DiBenedetto_Showalter_1981,Peszynska_Showalter_Yi_2009,Showalter_1975_nonlin,Showalter_1975_Sobolev2,Showalter_1972_rep,Carroll_Showalter_1976,Showalter_1970_SG,Showalter_1970_odd,Bohm_1992}.

We note that if one reintroduces a viscous term $\nu\triangle \bu$ to the right-hand side of  \eqref{MHD_V_mom} (with $\nu>0$), it is possible to make sense of Dirichlet (no-slip) boundary conditions $\bu=0$, and all of the results of the present work hold for such a case.  Observe that the results reported here are also valid in the whole space $\nR^3$ by employing the relevant analogue tools for treating the Navier--Stokes in the whole space.  We also remark that in the case $\nu>0$, $\mu>0$, and $\alpha>0$, one can consider the case of physical boundary conditions for \eqref{MHD_V}, that is, $\bu\big|_{\partial\Omega}=0$, $\mathbf{n}\cdot\B\big|_{\partial\Omega}=0$ and $\mathbf{n}\times(\nabla\times\B\big)|_{\partial\Omega}=0$, which are also known as no-slip, superconductor boundary conditions (see, e.g., \cite{Cao_Ettinger_Titi_2010,Duvaut_Lions_1972,Schnack_2009}).  With these boundary conditions, one can also prove that the system \eqref{MHD_V} enjoys global regularity with appropriate modifications to the methods employed in the present work, a subject of a forthcoming paper.  However, the higher-order regularity of solutions to \eqref{MHD_V} in the case of physical boundary conditions does not follow directly from the proofs below, and one would have to modify the functional spaces taking into consideration the presence of a physical boundary (see, e.g., \cite{Kukavica_Vicol_2009} and references therein).

System \eqref{MHD_V} was introduced and studied in the two-dimensional case in \cite{Oskolkov_1975_MHD}, where global well-posedness was proven under the assumption that $\bu_0\in V$, $\B_0\in H$.  The three-dimensional case was studied in \cite{Catania_Secchi_2009}, where global regularity was proven, assuming the initial data $\bu_0\in H^2(\Omega)$, $\B_0\in H^1(\Omega)$ and are divergence free.
Here, we relax the hypotheses of the theorems given in \cite{Catania_Secchi_2009} by requiring only that $\bu_0\in H^1(\Omega)$, $\B_0\in L^2(\Omega)$ and are divergence free for the existence of weak solutions, and $\bu_0, \B_0\in H^1(\Omega)$ for the existence of strong solutions.  Furthermore, we prove the uniqueness of strong solutions to \eqref{MHD_V}, a result which is stated, but not proven, in \cite{Catania_Secchi_2009}.  We also prove the higher-order regularity of \eqref{MHD_V}.
In Section \ref{sec:Pre}, we introduce some notation and preliminary results. In Section \ref{sec:Exist_Unique}, we establish the global existence of weak and strong solutions, and the uniqueness of strong solutions among the class of weak solutions.  We give a complete proof using the Galerkin method, and we justify rigorously the a priori estimates and the existence and uniqueness results.  In Section \ref{sec:Reg}, we establish higher-order regularity of strong solutions. In Section \ref{sec:conv_blow_up}, we show that strong solutions to the Voigt-regularized MHD equations \eqref{MHD_V} converge, as $\alpha \to 0$,  to strong solutions of the MHD equations \eqref{MHD} (with $\nu,\mu\geq0$) on the time interval of existence of the latter.  Furthermore, in Section \ref{sec:conv_blow_up}, we establish a blow-up criterion for solutions of the inviscid, irresistive MHD equations, which can be easily implemented in numerical simulations.

\section{Preliminaries}\label{sec:Pre}
In this section, we introduce some preliminary material and notations which are commonly used in the mathematical study of fluids, in particular in the study of the Navier-Stokes equations (NSE).  For a more detailed discussion of these topics, we refer to \cite{Constantin_Foias_1988,Temam_1995_Fun_Anal,Temam_2001_Th_Num,Foias_Manley_Rosa_Temam_2001}.

In this paper, we consider only periodic boundary conditions.  Our space of test functions is defined to be
\[\mathcal{V}:=\set{\vphi\in\mathcal{F}:\nabla\cdot\vphi=0\text{ and  }\int_\Omega\vphi(x)\,d\bx=0},\]
where $\mathcal{F}$ is the set of all three-dimensional vector-valued trigonometric polynomials with periodic domain $\Omega=\nT^3:=[0,1]^3$.
We denote by $L^p$ and $H^m$ the usual Lebesgue and Sobolev spaces over $\Omega$, and define $H$ and $V$ to be the closures of $\mathcal{V}$ in $L^2$ and $H^1$, respectively.  We define the inner products on $H$ and $V$ respectively by
\[(\bu,\bv)=\sum_{i=1}^3\int_\Omega u_iv_i\,d\bx
\sand
((\bu,\bv))=\sum_{i,j=1}^3\int_\Omega\pd{u_i}{x_j}\pd{v_i}{x_j}\,d\bx,
\]
and the associated norms $|\bu|=(\bu,\bu)^{1/2}$, $\|\bu\|=((\bu,\bu))^{1/2}$.  Note that $((\cdot,\cdot))$ is a norm due to the Poincar\'e inequality, \eqref{poincare} below.
  We denote by $V'$ the dual space of $V$.  The action of $V'$ on $V$ is denoted by $\ip{\cdot}{\cdot}\equiv \ip{\cdot}{\cdot}_{V'}$.

Let $X$ be a Banach space with dual space $X'$.  We denote by $L^p((a,b),X)$ the space of Bochner measurable functions $t\mapsto \bw(t)$, where $\bw(t)\in X$ for a.e. $t\in(a,b)$, such that the integral $\int_a^b\|\bw(t)\|_X^p\,dt$ is finite (see, e.g., \cite{Adams_Fournier_2003}). A similar convention is used for $C^k((a,b),X)$.  The space $C_w([0,T],X)$ is the subspace of $L^\infty([0,T],X)$ consisting of all functions which are weakly continuous, that is, all functions $\bw\in L^\infty([0,T],X)$ such that $\ip{\bw(t)}{\bv}$ is a continuous function for all $\bv\in X'$.  Note that here and below, we abuse notation slightly, writing $\bw(\cdot)$ for the map $t\mapsto \bw(t)$.  In the same vein, we often write the vector-valued function $\bw(\cdot,t)$ as $\bw(t)$ when $\bw$ is a function of both $\bx$ and $t$.  We stress that whenever we write $\bw$ satisfies $\frac{d\bw}{dt}\in L^1((a,b),X)$, we implicity mean that  $\bw:(a,b)\maps X$ is absolutely continuous in time with values in $X$ (see, e.g., \cite{Temam_2001_Th_Num}).

We denote by $P_\sigma:L^2\maps H$ the Leray-Helmholtz projection
operator and define the Stokes operator $A:=-P_\sigma\triangle$ with
domain $\mathcal{D}(A):=H^2\cap V$.  $A$ can be extended as a bounded linear operator $A:V\maps V'$, such that $\|A\bv\|_{V'}=\|\bv\|$ for all $\bv\in V$. It is known that $A^{-1}:H\maps \mathcal{D}(A)$ is a
positive-definite, self-adjoint, compact operator, and that there is an orthonormal basis $\set{\bw_i}_{i=1}^\infty$ of $H$ consisting of eigenvectors of $A$ corresponding to eigenvalues $\set{\lambda_i}_{i=1}^\infty$ such that $A\bw_j=\lambda_j\bw_j$ and
$0<\lambda_1\leq\lambda_2\leq\lambda_3\leq\cdots$ (see, e.g., \cite{Constantin_Foias_1988,Temam_1995_Fun_Anal,Temam_2001_Th_Num}) repeated according to their multiplicity.
Let $H_m:=\text{span}\set{\bw_1,\ldots,\bw_m}$, and let $P_m:H\maps H_m$ be the $L^2$ orthogonal projection onto $H_m$ with respect to $\set{\bw_i}_{i=1}^\infty$.  Notice that in the case of periodic boundary conditions, i.e., in the torus $\nT^3$, we have $A=-\triangle$, and $\lambda_1=(2\pi)^{-2}$ (see, e.g., \cite{Constantin_Foias_1988,Temam_1995_Fun_Anal}).  We have the continuous embeddings
\begin{equation}\label{embed}
 \mathcal{D}(A)\hookrightarrow V\hookrightarrow H\equiv H'\hookrightarrow V'.
\end{equation}
Moreover, by the Rellich-Kondrachov Compactness Theorem (see, e.g., \cite{Evans_1998,Adams_Fournier_2003}), these embeddings are compact.

It will be convenient to suppress the pressure term by applying the Leray-Helmholtz projection $P_\sigma$ and use the standard notation for the non-linearity,
\begin{equation}\label{Bdef}
 B(\bu,\bv):=P_\sigma((\bu\cdot\nabla)\bv)
\end{equation}
for $\bu,\bv\in\mathcal{V}$.  We list several important properties of $B$ which can be found for example in \cite{Constantin_Foias_1988, Foias_Manley_Rosa_Temam_2001, Temam_1995_Fun_Anal, Temam_2001_Th_Num}.  The proof of this lemma relies mainly on \eqref{Bdef} and inequalities of the type \eqref{Agmon1/2}-\eqref{poincare} below.
\begin{lemma}\label{B:prop}
The operator $B$ defined in \eqref{Bdef} is a bilinear form which can be extended as a continuous map $B:V\times V\maps V'$.  Furthermore, the following properties hold.

 \begin{enumerate}[(i)]
  \item For $\bu$, $\bv$, $\bw\in V$,
\begin{equation}\label{B:alt}
 \ip{B(\bu,\bv)}{\bw}_{V'}=-\ip{B(\bu,\bw)}{\bv}_{V'},
\sand
 \ip{B(\bu,\bv)}{\bv}_{V'}=0.
\end{equation}
\item We have the following estimates.
\begin{subequations}\label{B:est}
  \begin{align}
  \label{B:326}
|\ip{B(\bu,\bv)}{\bw}_{V'}|
&\leq C|\bu|^{1/2}\|\bu\|^{1/2}\|\bv\|\|\bw\|,
&\quad\forall\;\bu\in V, \bv\in V, \bw\in V,\\
\label{B:623}
|\ip{B(\bu,\bv)}{\bw}_{V'}|
&\leq C\|\bu\|\|\bv\||\bw|^{1/2}\|\bw\|^{1/2},
&\quad\forall\;\bu\in V, \bv\in V, \bw\in V,\\
\label{B:236}
|\ip{B(\bu,\bv)}{\bw}_{V'}|
&\leq C|\bu|\|\bv\|^{1/2}|A\bv|^{1/2}\|\bw\|,
&\quad\forall\;\bu\in H, \bv\in \mathcal{D}(A), \bw\in V,\\
\label{B:632}
|\ip{B(\bu,\bv)}{\bw}_{V'}|
&\leq C\|\bu\|\|\bv\|^{1/2}|A\bv|^{1/2}|\bw|,
&\quad\forall\;\bu\in V, \bv\in \mathcal{D}(A), \bw\in H,\\
\label{B:i22}
|\ip{B(\bu,\bv)}{\bw}_{V'}|
&\leq C\|\bu\|^{1/2}|A \bu|^{1/2}\|\bv\||\bw|,
&\quad\forall\;\bu\in \mathcal{D}(A), \bv\in V, \bw\in H,\\
\label{B:362s}
|\ip{B(\bu,\bv)}{\bw}_{\mathcal{D}(A)'}|
&\leq C|\bu|^{1/2}\|\bu\|^{1/2}| \bv||A\bw|,
&\quad \forall\;\bu\in V, \bv\in H, \bw\in\mathcal{D}(A),\\
\label{B:263}
|\ip{B(\bu,\bv)}{\bw}_{V'}|
&\leq C|\bu||A\bv||\bw|^{1/2}\|\bw\|^{1/2},
&\quad\forall\;\bu\in H, \bv\in \mathcal{D}(A), \bw\in V.
\end{align}
\end{subequations}
 \end{enumerate}
\end{lemma}
\noindent We also define the trilinear form $b:V\times V\times V\maps\nR$ by
\begin{equation}\label{b_tri}
b(\bu,\bv,\bw)
:= \ip{B(\bu,\bv)}{\bw}_{V'}.
\end{equation}
Here and below, $K_\alpha, K_{\alpha,\mu},C(\cdots)$, etc. denote generic constants which depend only upon the indicated parameters, and which may change from line to line.  Let us distinguish between $K$ and $C$.  $K$ will depend on some norm of the solutions, but $C$ will not be used for constants which depend on functions.

Next, we recall Agmon's inequality (see, e.g., \cite{Agmon_1965, Constantin_Foias_1988}).  For $\bu\in\mathcal{D}(A)$ we have
\begin{equation}\label{Agmon1/2}
 \|\bu\|_{L^\infty(\Omega)} \leq C\|\bu\|^{1/2}|A\bu|^{1/2}\:\:.
\end{equation}
We also have the Sobolev and Ladyzhenskaya inequalities in three-dimensions,
\begin{align}
\label{L3_interp}
\|\bu\|_{L^3}&\leq C|\bu|^{1/2}\|\bu\|^{1/2}
\\\label{L6_sobol}
\|\bu\|_{L^6}&\leq C\|\bu\|,
\end{align}
 for every $\bu\in V$.  Furthermore, for all $\bw\in V$, we have
the Poincar\'e inequality
\begin{equation}\label{poincare}
   |\bw|\leq\lambda_1^{-1/2} |\nabla\bw|
   =\lambda_1^{-1/2}\|\bw\|.
\end{equation}
Due to \eqref{poincare} and the elliptic regularity of the solutions to the Stokes equation (see, e.g., \cite{Constantin_Foias_1988,Temam_1995_Fun_Anal}), for $\bw\in \mathcal{D}(A)$, we have the norm equivalence
\begin{equation}\label{elliptic_reg}
   |A\bw|\cong\|\bw\|_{H^2}.
\end{equation}
Finally, we note a result of deRham (see, e.g., \cite{Wang_1993, Temam_2001_Th_Num}), which states that if $\bg$ is a locally integrable function (or more generally, a distribution), we have
\begin{equation}\label{deRham}
 \bg =\nabla p \text{ for some distribution $p$ iff } \ip{\bg}{\bv}=0\quad \forall \bv\in\mathcal{V}.
\end{equation}
This result is useful for recovering the pressure term as it is treated, for example, in \cite{Temam_2001_Th_Num}.

\section{Existence and Uniqueness of Solutions}\label{sec:Exist_Unique}

This section is devoted to stating and proving our main result.  As mentioned in the introduction, we relax the conditions of the results of \cite{Catania_Secchi_2009}, where it was assumed that $\bu_0\in \mathcal{D}(A)$, $\B_0\in V$ to derive the existence of a strong solution.  Here we define the notion of a weak solution to \eqref{MHD_V}, for which we only need to assume that $\bu_0\in V$, $\B_0\in H$ to show global existence without uniqueness.  As for a strong solution, we only need to assume $\bu_0, \B_0\in V$ to prove global existence, uniqueness, and continuous dependence on initial data.  We note that, although the major features of this model that allow for a proof of well-posedness were exploited formally in \cite{Catania_Secchi_2009}, the \textit{a priori} estimates can be sharpened.  This allows us to define the notion of weak solutions and to prove their global existence.  We also prove the uniqueness of strong solutions, which was stated without proof in  \cite{Catania_Secchi_2009}.  Furthermore, there are subtleties in passing to the limit due to the addition of the Voigt term $-\alpha^2\partial_t\triangle\bu$ in the momentum equation \eqref{MHD_V_mom} that one has to address in the rigorous proof.  Here, we give a fully rigorous derivation and justification of these estimates, as well as the passage to the limit.

In order to prove global existence, we use the Galerkin approximation procedure, based on the eigenfunctions of the Stokes operator.  The proof is broken into several steps.  First, we show that there exists a solution to the finite-dimensional approximating Galerkin problem which is bounded in the appropriate norms.  In particular, we show that the time derivatives of the sequence of approximating solutions are uniformly bounded in the appropriate spaces.  We then extract appropriate subsequences using the Banach-Alaoglu and Aubin Compactness Theorems (see, e.g., \cite{Constantin_Foias_1988},p. 68-71 or \cite{Robinson_2001,Temam_2001_Th_Num}), and pass to the limit to obtain a global solution to system \eqref{MHD_V}.  Finally, we argue that the solutions satisfy the initial conditions in the sense given in Definition \ref{def:wk_st} below.

\subsection{Existence of Weak Solutions}\label{sec:exist}

Before we begin, we rewrite \eqref{MHD_V} in functional form.  Applying $P_\sigma$ to \eqref{MHD_V} and using the notation introduced in Section \ref{sec:Pre}, we obtain the following system, which is equivalent to \eqref{MHD_V_proj} (see, e.g., \cite{Temam_2001_Th_Num} for showing the equivalence in the context of the Navier-Stokes equations)
\begin{subequations}\label{MHD_V_proj}
\begin{align}
 \frac{d}{dt}\pnt{\alpha^2 A \bu+\bu} &=B(\B,\B)-B(\bu,\bu),
\label{MHD_V_proj_mom}\\
 \frac{d}{dt}\B + \mu A \B&=B(\B,\bu)-B(\bu,\B),
\label{MHD_V_proj_mag}\\
 \B(0)=\B_0,\; \bu(0)&=\bu_0,
 \label{LHproj3}
\end{align}
\end{subequations}
where \eqref{MHD_V_proj_mom} is satisfied in the sense of $L^{4/3}((0,T),V')$, \eqref{MHD_V_proj_mag} is satisfied in the sense of $L^2((0,T),V')$, and \eqref{LHproj3} is satisfied in the sense of Definition \ref{def:wk_st} below.  Systems \eqref{MHD_V} and \eqref{MHD_V_proj} are equivalent, and one can recover the pressure terms $p$ and $q$ ($q\equiv0$) by using \eqref{deRham}, as it is done for the case of the Navier-Stokes equations (see, e.g., \cite{Duvaut_Lions_1972,Temam_2001_Th_Num}).

\begin{definition}\label{def:wk_st}
   Let $\bu_0\in V$, $\B_0\in H$.  We say that $(\bu,\B)$ is a \textit{weak solution} to \eqref{MHD_V_proj}, on the time interval $[0,T]$, if $\bu\in C([0,T],V)$, $\B\in L^2((0,T),V)\cap C_w([0,T],H)$, $\frac{d\bu}{dt}\in L^4((0,T),H)$, $\frac{d\B}{dt}\in L^2((0,T),V')$, and  furthermore, $(\bu,\B)$ satisfies \eqref{MHD_V_proj_mom} in the sense of $L^{4/3}([0,T],V')$ and \eqref{MHD_V_proj_mag} in the sense of $L^2([0,T],V')$.
   Furthermore, if $\bu_0,\B_0\in V$, we say that  $(\bu,\B)$ is a \textit{strong solution} to \eqref{MHD_V_proj} if it is a weak solution, and additionally, $\B\in L^2((0,T),\mathcal{D}(A))\cap C([0,T],V)$, $\frac{d\bu}{dt}\in C([0,T],V)$, and $\frac{d\B}{dt}\in L^2((0,T),H)$.
\end{definition}
With this definition, we are now ready to state and prove the following theorem.

\begin{theorem}\label{thm:weak}
Let $\bu_0\in V$, $\B_0\in H$. Then \eqref{MHD_V_proj} has a weak solution $(\bu,\B)$ for arbitrary  $T>0$.
\end{theorem}

\begin{proof}

Let $T>0$ be fixed.  Consider the finite dimensional Galerkin approximation of \eqref{MHD_V_proj}, based on the eigenfunctions of the operator $A$ (see Section \ref{sec:Pre}), given by the following system of ODEs in $H_m\times H_m$.
\begin{subequations}\label{irMHD_Gal}
\begin{align}
\label{irMHD_Gal_u}
 \frac{d}{dt}\pnt{\bu_m+\alpha^2 A \bu_m}+P_mB(\bu_m,\bu_m)
 &= P_mB(\B_m,\B_m),
\\\label{irMHD_Gal_B}
 \frac{d}{dt}\B_m +\mu A\B_m + P_mB(\bu_m,\B_m)&=P_mB(\B_m,\bu_m),
\\\label{irMHD_Gal_init}
 \B_m(0)=P_m\B_0,\; \bu_m(0)&=P_m\bu_0.
\end{align}
\end{subequations}
We look for a solution $\bu_m,\B_m\in C^1([0,T_m),H_m)$ of \eqref{irMHD_Gal}.  By applying the operator $(I+\alpha^2A)^{-1}$ to \eqref{irMHD_Gal_u}, we see that \eqref{irMHD_Gal} is equivalent to a system of the form $\byd=\mathbf{F}(\by)$, where $\mathbf{F}:H_m\times H_m\maps H_m\times H_m$ is a quadratic polynomial.  By classical ODE theory, this system has a unique solution on $[0,T_m)$ for some $T_m>0$.   Let $[0,T_m^{\max})$ be the maximal interval where existence and uniqueness of the solution of \eqref{irMHD_Gal} holds.


Next, we show that $T_m^{\text{max}}=\infty$.  Indeed, taking the inner product of \eqref{irMHD_Gal_u} with $\bu_m(t)$ and \eqref{irMHD_Gal_B} with $\B_m(t)$, for $t\in[0,T_m^{\max})$, and integrating by parts with respect to the spatial variable and using \eqref{B:alt}, we have,
\begin{subequations}\label{MHDL2ip}
\begin{align}
 \frac{1}{2}\frac{d}{dt}\pnt{\alpha^2 \|\bu_m\|^2+|\bu_m|^2} &=(B(\B_m,\B_m),\bu_m),
\label{MHDL2ip1}\\
 \frac{1}{2}\frac{d}{dt}|\B_m|^2 + \mu\|\B_m\|^2&=(B(\B_m,\bu_m),\B_m)=-(B(\B_m,\B_m),\bu_m).
\label{MHDL2ip2}
\end{align}
\end{subequations}
Adding \eqref{MHDL2ip1} and \eqref{MHDL2ip2} gives
\begin{align}\label{diff_en_Gal}
 \frac{1}{2}\frac{d}{dt}\pnt{\alpha^2 \|\bu_m\|^2+|\bu_m|^2+|\B_m|^2} &= -\mu\|\B_m\|^2\leq0.
\end{align}
Integrating the equality in \eqref{diff_en_Gal} in time, we obtain for $t\in [0,T_m^{\max})$
\begin{align}
&\notag\quad
 \alpha^2 \|\bu_m(t)\|^2+|\bu_m(t)|^2+|\B_m(t)|^2
 +2\mu\int_0^{t}\|\B_m(s)\|^2\,ds
\\&\notag
=\alpha^2 \|\bu_m(0)\|^2+|\bu_m(0)|^2+|\B_m(0)|^2
\\&\leq\label{u_H1_Gron}
(K_\alpha^1)^2:=\alpha^2 \|\bu_0\|^2+|\bu_0|^2+|\B_0|^2.
\end{align}
This bound, together with the fact that the vector field $\mathbf{F}(\by)$ in \eqref{irMHD_Gal} is a quadratic polynomial, imply  that $T_m^{\max}=\infty$.  Furthermore, we see from \eqref{u_H1_Gron} that for fixed but arbitrary $T>0$, we have
\begin{subequations}
   \begin{align}
\label{u_LiV}
&\text{$\bu_m$ is bounded  in $L^\infty([0,T],V)$,}\\
\label{B_LiH_L2V}
&\text{$\B_m$ is bounded in $L^\infty([0,T],H)\cap L^2([0,T],V)$,}
\end{align}
\end{subequations}
uniformly with respect to $m$.

As mentioned at the beginning of this section, our goal is to extract subsequences of $\set{\bu_m}$ and $\set{\B_m}$ which converge in $L^2((0,T),H)$ by using the Aubin Compactness Theorem (see, e.g., \cite{Constantin_Foias_1988},p. 68-71 or \cite{Robinson_2001,Temam_2001_Th_Num}).  To satisfy the hypotheses of Aubin's theorem, we show that $\frac{d\bu_m}{dt}$ is uniformly bounded in $L^{4}((0,T),H)\hookrightarrow L^{2}((0,T),V')$, and that $\frac{d\B_m}{dt}$ is uniformly bounded in $L^{2}((0,T),V')$, with respect to $m$.    Using  equation \eqref{MHD_V_proj_mom}, we have from \eqref{B:362s},
\begin{align}
&\quad\notag
\norm{(I+\alpha^2A)\frac{d\bu_m}{dt}}_{\mathcal{D}(A)'}
\\&\notag \leq
\|P_mB(\B_m,\B_m)\|_{\mathcal{D}(A)'}+\|P_mB(\bu_m,\bu_m)\|_{\mathcal{D}(A)'}
\\\notag &\leq
C
|\B_m|^{3/2}\|\B_m\|^{1/2}
+ C
|\bu_m|^{3/2}\|\bu_m\|^{1/2}
\\&\leq\label{u_t_Est}
C(K_\alpha^1)^{3/2}\|\B_m\|^{1/2}+ (K_\alpha^1)^{2}\alpha^{-1/2},
\end{align}
where we have used \eqref{u_H1_Gron}.  Estimating differently, we have
\begin{align}
&\quad\notag
\norm{(I+\alpha^2A)\frac{d\bu_m}{dt}}_{V'}
\\&\notag \leq
\|P_mB(\B_m,\B_m)\|_{V'}+\|P_mB(\bu_m,\bu_m)\|_{V'}
\\&=\notag
\sup_{\|\bw\|=1}\pair{B(\B_m,\B_m)}{P_m\bw}
+ \sup_{\|\bw\|=1}\pair{B(\bu_m,\bu_m)}{P_m\bw}
\\\notag &\leq
C\sup_{\|\bw\|=1}|\B_m|^{1/2}\|\B_m\|^{3/2}\|\bw\|
+ C\sup_{\|\bw\|=1}|\bu_m|^{1/2}\|\bu_m\|^{3/2}\|\bw\|
\\&\leq\label{u_t_Est_2}
C(K_\alpha^1)^{1/2}\|\B_m\|^{3/2}+ (K_\alpha^1)^{2}\alpha^{-3/2},
\end{align}
Thus, due to \eqref{u_H1_Gron}, the right-hand side of \eqref{u_t_Est} is uniformly bounded in $L^{4}(0,T)$, and right-hand side of \eqref{u_t_Est_2} is uniformly bounded in $L^{4/3}(0,T)$, and hence, $(I+\alpha^2A)\frac{d\bu_m}{dt}$ is uniformly bounded in $L^{4}([0,T],\mathcal{D}(A)')\cap L^{4/3}([0,T],V')$  with respect to $m$.
By inverting the operator $(I+\alpha^2A)$, we have
\begin{align}\label{u_t_L4V}
\frac{d\bu_m}{dt}
\text{ is bounded in }
L^{4}([0,T],H)
\text{ and }
L^{4/3}([0,T],V),
\end{align}
uniformly with respect to $m$.

Next, we estimate $\frac{d\B_m}{dt}$.  From equation \eqref{MHD_V_proj_mag} we have, thanks to \eqref{B:326} and \eqref{B:623},
\begin{align}
\norm{\frac{d\B_m}{dt}}_{V'}
&\notag\leq
\|P_mB(\B_m,\bu_m)\|_{V'}+\|P_mB(\bu_m,\B_m)\|_{V'} + \mu\|A\B_m\|_{V'}
\\\notag &\leq
C|\B_m|^{1/2}\|\B_m\|^{1/2}\|\bu_m\|+\mu\|\B_m\|
\\\label{B_t_Est}&\leq
C(K^1_\alpha)^{3/2}\alpha^{-1}\|\B_m\|^{1/2}+\mu\|\B_m\|,
\end{align}
where the last estimate is due to \eqref{u_H1_Gron}.  Thus, by virtue of \eqref{B_LiH_L2V}, it follows that
\begin{align}\label{B_t_L2V'}
\text{
$\frac{d\B_m}{dt}$ is bounded in $L^2([0,T],V')$,
}
\end{align}
uniformly with respect to $m$.

We have shown in \eqref{u_LiV} that $\bu_m$ is uniformly bounded in $L^\infty([0,T],V)\hookrightarrow L^2([0,T],V)$, that $\frac{d\bu_m}{dt}$ is uniformly bounded in $L^{4}([0,T],H)\hookrightarrow L^{2}([0,T],V')$, that  $\B_m$ is uniformly bounded in $L^\infty([0,T],H)\cap L^2([0,T],V)$, and that $\frac{d\B_m}{dt}$ is uniformly bounded in $L^2([0,T],V')$.  Thus, by the Aubin Compactness Theorem (see, e.g., \cite{Constantin_Foias_1988},p. 68-71 or \cite{Robinson_2001,Temam_2001_Th_Num}), there exists a subsequence of $(\B_m,\bu_m)$ (which we relabel as  $(\B_m,\bu_m)$, if necessary) and elements $\B,\bu\in L^2([0,T],H)$ such that
\begin{subequations}\label{st_conv}
\begin{align}
\label{st_conv_B}
\B_m\maps\B&\quad\text{strongly in }L^2([0,T],H),
\\
\label{st_conv_u}
\bu_m\maps\bu&\quad\text{strongly in }L^2([0,T],H).
\end{align}
\end{subequations}
Furthermore, using \eqref{u_LiV}, \eqref{B_LiH_L2V}, \eqref{u_t_L4V}, \eqref{B_t_L2V'} and the Banach-Alaoglu Theorem,  we can pass to additional subsequences if necessary (which we again relabel as  $(\B_m,\bu_m)$), to show that, in fact, $\bu\in L^\infty([0,T],V)$, $\B\in L^\infty([0,T],H)\cap L^2([0,T],V)$, $\frac{d}{dt}\bu\in L^4([0,T],V)\cap L^{4/3}([0,T],H)$, $\frac{d}{dt}\B\in L^2([0,T],V')$, and
\begin{subequations}\label{wk_conv}
\begin{align}
\label{wk_conv_L2V}
\B_m\rightharpoonup\B \sand \bu_m\rightharpoonup\bu&\quad\text{weakly in }L^2([0,T],V),
\\\label{wk_conv_LiH}
\B_m\rightharpoonup\B \sand \bu_m\rightharpoonup\bu&\quad\text{weak-$*$ in }L^\infty([0,T],H),
\\\label{wk_conv_LiV}
\bu_m\rightharpoonup\bu&\quad\text{weak-$*$ in }L^\infty([0,T],V),
\\\label{wk_conv_u_t}
\frac{d}{dt}\bu_m\rightharpoonup\frac{d}{dt}\bu&\quad\text{weak-$*$ in }L^4([0,T],H)\text{ and }L^{4/3}([0,T],V),
\\\label{wk_conv_B_t}
\frac{d}{dt}\B_m\rightharpoonup\frac{d}{dt}\B&\quad\text{weak-$*$ in }L^2([0,T],V').
\end{align}
\end{subequations}
as $m\maps\infty$.

Let $k$ be fixed, and take $m\geq k$.  Let $\bw\in C^1([0,T],H_k)$ with $\bw(T)=0$ be arbitrarily given.  By taking the inner product of \eqref{irMHD_Gal} with $\bw$, integrating on $[0,T]$, and using integration by parts, we have
\begin{subequations}\label{IR_MHD_Gal_Int}
\begin{align}
&\quad \label{IR_MHD_Gal_Int_u}
-(\bu_m(0),\bw(0))-\alpha^2((\bu_m(0),\bw(0)))
\\&\notag\quad
-\int_{0}^{T}(\bu_m(t),\bw'(t))\,dt
+\alpha^2\int_{0}^{T} ((\bu_m(t),\bw'(t)))\,dt
\\&\notag =
\int_{0}^{T}(B(\B_m(t),\B_m(t),P_m\bw(t))\,dt
-\int_0^T(B(\bu_m(t),\bu_m(t)),P_m\bw(t))\,dt,
\\&\quad\label{IR_MHD_Gal_Int_B}
-(\B_m(0),\bw(0))-\int_{0}^{T}(\B'_m(t),\bw(t))\,dt
+\mu\int_{0}^{T}((\B_m(t),\bw(t)))\,dt
\\&\notag =
\int_{0}^{T}(B(\B_m(t),\bu_m(t)),P_m\bw(t))\,dt
-\int_0^T(B(\bu_m(t),\B_m(t)),P_m\bw(t))\,dt.
\end{align}
\end{subequations}
We show each of the terms in \eqref{IR_MHD_Gal_Int} converges to the appropriate limit, namely, we will find that equations \eqref{IR_MHD_Gal_Int} hold with $\set{\B_m,\bu_m,P_m}$ replaced by $\set{\B,\bu,I}$, where $I$ is the identity operator.  First, thanks to \eqref{wk_conv_L2V}, we have
\begin{align*}
   \mu\int_{0}^{T}((\B_m(t),\bw(t)))\,dt
   &\maps
   \mu\int_{0}^{T}((\B(t),\bw(t)))\,dt,\\
\int_{0}^{T}(\bu_m(t),\bw'(t))\,dt
&\maps
\int_{0}^{T}(\bu(t),\bw'(t))\,dt
,\\
\alpha^2\int_{0}^{T} ((\bu_m(t),\bw'(t)))\,dt
&\maps
\alpha^2\int_{0}^{T} ((\bu(t),\bw'(t)))\,dt
,\\
\int_{0}^{T}(\B_m(t),\bw'(t))\,dt
&\maps
\int_{0}^{T}(\B(t),\bw'(t))\,dt
.
\end{align*}

Next, we must show the convergence of the trilinear forms.  We will only show the convergence of one of them, as the rest are similar (see, e.g., \cite{Constantin_Foias_1988,Temam_2001_Th_Num} for similar arguments in the case of the Navier-Stokes equations).  Namely, we will show that
\begin{equation*}
   I(m):=\int_{0}^{T}(B(\bu_m(t),\bu_m(t)),P_m\bw(t))\,dt-\int_{0}^{T}\ip{B(\bu(t),\bu(t))}{\bw(t))}_{V'}\,dt\maps0.
\end{equation*}
To this end, let
\begin{align*}
   I_1(m)&:= \int_{0}^{T}\ip{B(\bu_m(t)-\bu(t)}{P_m\bw(t)),\bu_m(t)}_{V'}\,dt,\\
   I_2(m)&:= \int_{0}^{T}\ip{B(\bu(t),P_m\bw(t))}{\bu_m(t)-\bu(t)}_{V'}\,dt,
\end{align*}
and note that $I(m)= I_1(m)+I_2(m)$, where we have used \eqref{B:alt}.  Since $\bw\in C^1([0,T],H_k)$ and $k\geq m$, we have $P_m\bw=\bw$.  Thus, thanks to \eqref{B:236}, \eqref{B:632}, and \eqref{u_LiV}, a simple application of H\"older's inequality and \eqref{st_conv_u} shows that $I_1(m)\maps0$ and $I_2(m)\maps0$, and thus $I(m)\maps0$ for $\bw\in C^1([0,T],H_k)$.
As mentioned above, similar arguments hold for the other tri-linear terms.  Note that $\bu_m(0):=P_m\bu_0\maps \bu_0$ in $V$ and $\B_m(0):=P_m\B_0\maps \B_0$ in $H$.  Thus, passing to the limit as $m\maps\infty$ in \eqref{IR_MHD_Gal_Int}, we have for all $\bw\in C^1([0,T],H_k)$ with $\bw(T)=0$, 
\begin{subequations}\label{IR_MHD_Gal_lim}
\begin{align}
&\quad \label{IR_MHD_Gal_lim_u}
-(\bu_0,\bw(0))-\alpha^2((\bu_0,\bw(0)))
\\&\notag\quad
-\int_{0}^{T}(\bu(t),\bw'(t))\,dt
+\alpha^2\int_{0}^{T} ((\bu(t),\bw'(t)))\,dt
\\&\notag =
\int_{0}^{T}(B(\B(t),\B(t),\bw(t))\,dt
-\int_0^T(B(\bu(t),\bu(t)),\bw(t))\,dt,
\\&\quad\label{IR_MHD_Gal_lim_B}
-(\B_0,\bw(0))-\int_{0}^{T}(\B(t),\bw'(t))\,dt
+\mu\int_{0}^{T}((\B(t),\bw(t)))\,dt
\\&\notag =
\int_{0}^{T}(B(\B(t),\bu(t)),\bw(t))\,dt
-\int_0^T(B(\bu(t),\B(t)),\bw(t))\,dt.
\end{align}
\end{subequations}
Since $C^1([0,T],H_k)$ is dense in $C^1([0,T],V)$, we use \eqref{B:326} and the facts that $\bu\in L^\infty((0,T),V)$ and $\B\in L^\infty((0,T),H)\cap L^2((0,T),V)$ to show that \eqref{IR_MHD_Gal_lim} holds for all $\bw\in C^1([0,T],V)$ with $\bw(T)=0$.  In particular, \eqref{MHD_V_proj_mom} and \eqref{MHD_V_proj_mag} are satisfied by $(\bu,\B)$ in the sense of $V'$, where the time derivatives are taken in the sense of distributions on $(0,T)$.  Allowing \eqref{MHD_V_proj_mom} and \eqref{MHD_V_proj_mag} to act on $\bw$ and comparing with \eqref{IR_MHD_Gal_lim}, one finds that $\bu(0)+\alpha^2A\bu(0)=\bu_0+\alpha^2A\bu_0$ and $\B(0)=\B_0$ (see, e.g., \cite[p. 195]{Temam_2001_Th_Num}).  By inverting $I+\alpha^2A$, we then have $\bu(0)=\bu_0$.

We must now show that $\bu$ and $\B$ satisfy the requirements for continuity in time in Definition \eqref{def:wk_st}.  Taking the action of \eqref{IR_MHD_Gal_lim_B} with an arbitrary $\bv\in \mathcal{V}$ and integrating in time, we obtain, for $a.e.$ $t_0,t_1\in[0,T]$,
\begin{align}
&\quad\label{IR_MHD_cts_B}
(\B(t_1)-B(t_0),\bv)
+\mu\int_{t_0}^{t_1}((\B(t),\bv))\,dt
\\&\notag =
\int_{t_0}^{t_1}\ip{B(\B(t),\bu(t))}{\bv}_{V'}\,dt
-\int_{t_0}^{t_1}\ip{B(\bu(t),\B(t))}{\bv}_{V'}\,dt.
\end{align}
Since the integrands are in $L^1((0,T))$, we see from \eqref{IR_MHD_cts_B} that $\B\in C_w([0,T],\mathcal{V})$ by sending $t_1\maps t_0$.  By the density of $\mathcal{V}$ in $H$ and the fact that $\B\in L^\infty([0,T],H)$, a simple application of the triangle inequality shows that $\B\in C_w([0,T],H)$.  Next, by \eqref{u_t_L4V}, we have $\frac{d}{dt}\bu\in L^4([0,T],H)\hookrightarrow L^2([0,T],V')$.  Since we also have $\bu\in L^2([0,T],V)$, it follows that $\bu\in C([0,T],H)$ (see, e.g., \cite[Corollary 7.3]{Robinson_2001}).
We have now shown that $\bu$ and $\B$ satisfy all the requisite conditions of Definition \ref{def:wk_st}, and therefore we have proven the global existence of weak solutions.
\end{proof}

\subsection{Existence of Strong Solutions}

Next, we prove the existence of strong solutions, assuming $\B_0\in V$.  Uniqueness of strong solutions among the class of weak solutions, and also higher-order regularity will be proven in the next two sections.

\begin{theorem}\label{thm:strong}
Let $\bu_0,\B_0\in V$. Then for every $T>0$, \eqref{MHD_V_proj} has a strong solution $(\bu,\B)$ on $[0,T]$.
\end{theorem}
\begin{proof}
Taking the inner product of \eqref{irMHD_Gal_B} with $A \B_m$, we use \eqref{B:326} and \eqref{B:i22} to obtain,
\begin{align}
&\notag\quad
\frac{1}{2}\frac{d}{dt}\|\B_m\|^2+\mu|A\B_m|^2
\\\notag &=
-(B(\B_m, \bu_m), A\B_m)+(B(\bu_m, \B_m), A\B_m)
\\\notag &\leq
C\|\B_m\|^{1/2}|A\B_m|^{1/2}\|\bu_m\||A\B_m|
+C\|\bu_m\|\|\B_m\|^{1/2}|A\B_m|^{1/2}|A\B_m|
\\\label{B_H1} &\leq
K_\alpha^1\alpha^{-1}\|\B_m\|^{1/2}|A\B_m|^{3/2},
\end{align}
since $\|\bu_m\|$ is uniformly bounded by \eqref{u_H1_Gron}.  Due to Young's inequality, we have
\begin{equation}\label{B_young}
\alpha^{-1}K_\alpha^1\|\B_m\|^{1/2}|AB_m|^{3/2}\leq
K_{\alpha,\mu}^2\|\B_m\|^{2}+\frac{\mu}{2}|A\B_m|^{2},
\end{equation}
where $K_{\alpha,\mu}^2:=C(\alpha^{-1}K_\alpha^1)^4\mu^{-3}$.  Therefore, combining \eqref{B_H1} and \eqref{B_young}, we have
\begin{align}\label{Bm_diff}
\frac{1}{2}\frac{d}{dt}\|\B_m\|^2+\frac{1}{2}\mu|A\B_m|^2
\leq K_{\alpha,\mu}^2\|\B_m\|^2.
\end{align}
Integrating \eqref{Bm_diff} on $[0,t]$ gives
   \begin{align}
   \|\B_m(t)\|^2+\mu\int_0^t|A\B_m(s)|^2\,ds
   &\leq \notag
   \|\B_m(0)\|^2+2K_{\alpha,\mu}^2\int_0^t\|\B_m(s)\|^2\,ds
   \\&\leq\label{Bm_L2H2}
   \|\B_0\|^2+K_{\alpha,\mu}^2\frac{K_\alpha^1}{\mu}:=K_{\alpha,\mu}^3
   \end{align}
due to \eqref{u_H1_Gron}.  Since we are now assuming $\B_0\in V$, \eqref{Bm_L2H2} implies
\begin{align}\label{B_LiV_L2H2}
\text{
$\B_m$ is bounded in $L^\infty([0,T],V)\cap L^2([0,T],\mathcal{D}(A))$.}
\end{align}
uniformly with respect to $m$.
Furthermore, recalling \eqref{u_t_Est}, the improved bound \eqref{B_LiV_L2H2} now yields
\begin{align}\label{u_t_LiV}
\frac{d\bu_m}{dt}\text{ is bounded in }L^{\infty}([0,T],V).
\end{align}
uniformly with respect to $m$.

Next, we estimate $\frac{d\B_m}{dt}$.  From \eqref{B:326}, \eqref{Agmon1/2}, \eqref{poincare}, \eqref{MHD_V_proj_mag}, and \eqref{u_H1_Gron}, we have,
\begin{align}
\left|\frac{d\B_m}{dt}\right|
&\notag \leq
|B(\B_m,\bu_m)|+|B(\bu_m,\B_m)| + \mu|A\B_m|
\\\notag &=
\sup_{|\bw|=1}\pair{B(\B_m,\bu_m)}{\bw}
+ \sup_{|\bw|=1}\pair{B(\bu_m,\B_m)}{\bw}
+\mu|A\B_m|
\\\notag &\leq
\sup_{|\bw|=1}\|\B_m\|_{L^\infty}|\nabla\bu_m||\bw|
+ \sup_{|\bw|=1}|\bu_m|^{1/2}\|\bu_m\|^{1/2}|A\B_m||\bw|
+ \mu|A\B_m|
\\\notag &\leq
C\|\B_m\|^{1/2}|A\B_m|^{1/2}\|\bu_m\|
+ \mu|A\B_m|
\\\notag&\leq
CK_\alpha^1(\alpha^{-1} + \mu)|A\B_m|.
\end{align}
Thanks to this and \eqref{B_LiH_L2V}
\begin{align}
\label{B_t_L2H}
\frac{d\B_m}{dt}\text{ is  bounded in }L^{2}([0,T],H),
\end{align}
uniformly with respect to $m$.  Now, from the proof of Theorem \ref{thm:weak}, we already know that there exists a weak solution $(\bu,\B)$ of \eqref{MHD_V_proj} such that  $\B_{m'}\maps\B$ in $L^\infty([0,T],H)$ and $L^2([0,T],V)$ for some subsequence $\set{\B_{m'}}$.  Thanks to \eqref{B_LiV_L2H2} and \eqref{B_t_L2H}, we may apply the Aubin Compactness Theorem (see, e.g., \cite{Constantin_Foias_1988},p. 68-71 or \cite{Robinson_2001,Temam_2001_Th_Num}) to extract a subsequence (relabeled as  $(\bu_m,B_m)$) such that
\begin{align}
   \B_m\maps \B \text{ strongly in }L^2([0,T],V).
\end{align}

Using  \eqref{B_LiV_L2H2}, the Banach-Alaoglu Theorem, and the uniqueness of limits, we may pass to additional subsequences if necessary to show that $\B\in L^\infty((0,T),V)\cap L^2((0,T),\mathcal{D}(A))$. It is easy to see from \eqref{B_t_L2H} that  $\frac{d\B}{dt}\in L^2((0,T),H)$.   Thus, we must have $\B\in C([0,T],V)$ (see, e.g., \cite[Corollary 7.3]{Robinson_2001}).   Finally, since $\bu,\B\in C([0,T],V)$, it follows easily that the right-hand side of \eqref{MHD_V_proj_mom} is in $C([0,T],V')$.  Inverting $(I+\alpha^2A)$ shows that $\frac{d\bu}{dt}\in C([0,T],V)$.  Therefore, we have shown the existence of a strong solution to \eqref{MHD_V_proj}.
\end{proof}

\subsection{Uniqueness of Strong Solutions and Their Continuous Dependence On Initial Data}
In this section, we prove the uniqueness of strong solutions among the class of weak solutions.  As mentioned in the introduction, the uniqueness of (strong) solutions is stated in \cite{Catania_Secchi_2009}, but no proof is given.  We begin with a lemma, which is reminiscent of the Lions-Magenes Lemma (see, e.g., \cite{Lions_Magenes_1972_II,Temam_2001_Th_Num}).  We use similar ideas to those in \cite[Lemma 1.2, p. 176]{Temam_2001_Th_Num}.

\begin{lemma}\label{lemma:lions_lemma_almost}
   Let $\bv\in C((0,T),H)$ and $\frac{d}{dt}\bv\in L^p((0,T),H)$ for some $p\in[1,\infty]$.  Then the following inequality holds in the in distribution sense on $(0,T)$.
   \begin{align}\label{lions_lemma_almost}
      \frac{d}{dt}|\bv|^2 = 2\pair{\frac{d}{dt}\bv}{\bv}.
   \end{align}
   Moreover, $|\bv|^2$ is absolutely continuous.
\end{lemma}
\begin{proof}
   First, we note that \eqref{lions_lemma_almost} makes sense, due to the fact that $t\mapsto |\bv(t)|^2$ and $t\mapsto (\frac{d}{dt}\bv(t),\bv(t))$ are both elements of $L^1([0,T])$.  Let us write $\widetilde{\bv}$ for the function which is equal to $\bv$ on $[0,T]$ and equal to $0$ on $\nR\setminus[0,T]$.
   By a standard mollification process, we can find a sequence of functions $\set{\bv_k}_{k\in\nN}$ in $C^\infty([0,T])$ such that $\bv_k\maps\bv$ in $L^\infty_{\text{loc}}((0,T),H)$ and $\frac{d}{dt}\bv_k\maps \frac{d}{dt}\bv$ in $L^p_{\text{loc}}((0,T),H)$.  Clearly, equality \eqref{lions_lemma_almost} holds for $\bv_k$, and also $|\bv_k|^2\maps |\bv|^2$ in $L^1_{\text{loc}}((0,T))$.  We also have for $0<t_1<t_2<T$,
   \begin{align*}
      &\int_{t_1}^{t_2}\abs{\pair{\tfrac{d}{dt}\bv_k}{\bv_k}-\pair{\tfrac{d}{dt}\bv}{\bv}}\,dt
      \leq
         \int_{t_1}^{t_2}\abs{\pair{\tfrac{d}{dt}\bv_k-\tfrac{d}{dt}\bv}{\bv_k}}
            +\abs{\pair{\tfrac{d}{dt}\bv}{\bv_k-\bv}}\,dt
      \\&\leq
         C\|\tfrac{d}{dt}\bv_k-\tfrac{d}{dt}\bv\|_{L^p((t_1,t_2))}\|\bv_k\|_{L^\infty((t_1,t_2))}
            +\|\tfrac{d}{dt}\bv\|_{L^p((t_1,t_2))}\|\bv_k-\bv\|_{L^\infty((t_1,t_2))}.
   \end{align*}
   Thus, $\pair{\tfrac{d}{dt}\bv_k}{\bv_k}\maps \pair{\tfrac{d}{dt}\bv}{\bv}$ in $L^1_{\text{loc}}((0,T))$ as well, and so \eqref{lions_lemma_almost} holds in the scalar distribution sense.  Furthermore, since the right-hand side of \eqref{lions_lemma_almost} is integrable, we have $|\bv|^2\in W^{1,1}((0,T))$, and so $|\bv|^2$ is absolutely continuous in time.
\end{proof}

\begin{theorem}[Uniqueness and Continuous Dependence On Initial Data]\label{thm:unique}
Let \\$(\bu^1,\B^1)$ be a strong solution to \eqref{MHD_V} with initial data $\bu_0^1$, $\B_0^1\in V$ and let $(\bu^2,\B^2)$ be a weak solution with initial data $\bu_0^2$, $\B_0^2\in V$.  Let us write $\delta\bv :=\bv^1-\bv^2$ for two arbitrary, consecutively labeled vectors $\bv^1$ and $\bv^2$.  We have,
\begin{align}\label{uniqueness_ineq}
&\quad
|\delta\bu (t)|^2+\alpha^2\|\delta\bu (t)\|^2+ |\delta\B (t)|^2
+\mu\int_0^t\|\delta\B (s)\|^2e^{K^4_{\alpha,\mu}(t-s)}\,ds
\\&\leq\notag
\pnt{|\delta\bu _0|^2+\alpha^2\|\delta\bu _0\|^2+ |\delta\B _0|^2} e^{K_{\alpha,\mu}^4t}
\end{align}
$K_{\alpha,\mu}^4=K_{\alpha,\mu}^4(K_{\alpha,\mu}^1, K_{\alpha,\mu}^3, \alpha,\mu)$ is a positive constant.  In particular, if $\bu_0^1\equiv \bu_0^2$ and $\B_0^1\equiv \B_0^2$ in the sense of $H$, then $\bu^1\equiv \bu^2$ and $\B^1\equiv \B^2$ in the sense of $L^1([0,T],H)$.
\end{theorem}

\begin{proof}
  Let us denote $\delta\bu :=\bu^1-\bu^2$ and $\delta\B :=\B^1-\B^2$.  Subtracting, we get an equation for $\delta\bu$:
\begin{equation}\label{B1B2eq}
(I+\alpha^2A)\frac{d\delta\bu }{dt} = (\delta\B \cdot\nabla)\B^1+ (\B^2\cdot\nabla)\delta\B
-(\delta\bu \cdot\nabla)\bu^1-(\bu^2\cdot\nabla)\delta\bu.
\end{equation}
Applying $(I+\alpha^2A)^{-1/2}$ to both sides of \eqref{B1B2eq} we obtain
\begin{align}
&\quad\label{B1B2halved}
(I+\alpha^2A)^{1/2}\frac{d\delta\bu }{dt}
\\&=\notag
(I+\alpha^2A)^{-1/2}\pnt{(\delta\B \cdot\nabla)\B^1+ (\B^2\cdot\nabla)\delta\B
-(\delta\bu \cdot\nabla)\bu_1-(\bu_2\cdot\nabla)\delta\bu }.
\end{align}
Since $\delta\bu\in C([0,T],V)$ and $ \frac{d\delta\bu}{dt}\in L^{4/3}([0,T],V)$, and therefore $(I+\alpha^2A)^{1/2}\frac{d\delta\bu}{dt}\in L^{4/3}([0,T],H)$, we may take the inner product of \eqref{B1B2halved} with $(I+\alpha^2A)^{1/2}\delta\bu \in C([0,T],H)$.
Thanks to Lemma \ref{lemma:lions_lemma_almost},
we have $\pair{(I+\alpha^2A)^{1/2}\frac{d}{dt}\delta\bu}{(I+\alpha^2A)^{1/2}\delta\bu}=\frac{1}{2}\frac{d}{dt}\abs{(I+\alpha^2A)^{1/2}\delta\bu}^2$.  After using \eqref{B:alt}, we arrive at
\begin{equation}\label{u1-u2}
\frac{1}{2}\frac{d}{dt}\abs{(I+\alpha^2A)^{1/2}\delta\bu}^2
=b(\delta\B ,\B^1,\delta\bu )+b(\B^2,\delta\B ,\delta\bu )-b(\delta\bu ,\bu^1,\delta\bu).
\end{equation}
Arguing in a similar way to the standard Navier-Stokes theory (in particular, using the Lions-Magenes Lemma \cite{Lions_Magenes_1972_II,Temam_2001_Th_Num} rather than Lemma \ref{lemma:lions_lemma_almost}, see, e.g., \cite{Constantin_Foias_1988,Temam_2001_Th_Num,Robinson_2001}), we derive the following equation for $\delta\B$.
\begin{equation}\label{B1-B2}
\frac{1}{2}\frac{d}{dt}|\delta\B |^2 + \mu\|\delta\B \|^2
=
b(\delta\B ,\bu^1,\delta\B )+b(\B^2,\delta\bu ,\delta\B )-b(\delta\bu ,\B^1,\delta\B).
\end{equation}
Since $b(\B^2,\delta\bu ,\delta\B ) = -b(\B^2,\delta\B ,\delta\bu )$ by \eqref{B:alt}, we may add equations \eqref{u1-u2} and \eqref{B1-B2} to obtain
{\allowdisplaybreaks
\begin{align*}
&\quad
 \frac{1}{2}\frac{d}{dt}
\pnt{\abs{(I+\alpha^2A)^{1/2}\delta\bu}^2+ |\delta\B |^2}+ \mu\|\delta\B \|^2
\\&=
b(\delta\B ,\B^1,\delta\bu )-b(\delta\bu ,\bu^1,\delta\bu)+b(\delta\B ,\bu^1,\delta\B )-b(\delta\bu ,B^1,\delta\B)
 \\&\leq 
 C\|\delta\B  \|\|\B^1\|\| \delta\bu \|
+C|\delta\bu|^{1/2}\|\delta\bu  \|^{3/2}\|\bu^1\|
\\&\quad
+C|\delta\B |^{1/2}\| \delta\B\|^{3/2}\|\bu^1\|
+C\|\delta\bu  \|\|\B^1\||\delta\B|^{1/2}\| \delta\B \|^{1/2}
\\&\leq 
 K_{\alpha,\mu}^3\|\delta\B  \|\| \delta\bu \|
+K_{\alpha,\mu}^1|\delta\bu|^{1/2}\|\delta\bu  \|^{3/2}
\\&\quad
+K_{\alpha,\mu}^1|\delta\B |^{1/2}\| \delta\B\|^{3/2}
+K_{\alpha,\mu}^3\|\delta\bu  \||\delta\B|^{1/2}\| \delta\B \|^{1/2}
\\&\leq 
\frac{1}{2}K_{\alpha,\mu}^4\pnt{|\delta\bu |^2+\alpha^2\|\delta\bu \|^2+ |\delta\B |^2}+  \frac{\mu}{2}\|\delta\B \|^2,
\end{align*}
}
where $K_{\alpha,\mu}^4=K_{\alpha,\mu}^4(K_{\alpha,\mu}^1, K_{\alpha,\mu}^3, \alpha,\mu)$ is a positive constant.  Here, we have used Young's inequality several times, as well as \eqref{B:326}, \eqref{B:623}, \eqref{poincare}, \eqref{u_H1_Gron}, and \eqref{Bm_L2H2} (which hold in the limit as $m\maps\infty$ for $\B^i, \bu^i$ by properties of weak convergence).  Therefore,
\begin{equation*}
\frac{d}{dt}
\pnt{\abs{(I+\alpha^2A)^{1/2}\delta\bu}^2+ |\delta\B |^2}+ \mu\|\delta\B \|^2
\leq
K_{\alpha,\mu}^4\pnt{|\delta\bu |^2+\alpha^2\|\delta\bu \|^2+ |\delta\B |^2}.
\end{equation*}
Using Gr\"onwall's inequality together with the identity $\abs{(I+\alpha^2A)^{1/2}\delta\bu}^2=|\delta\bu|^2+\alpha^2\|\delta\bu\|^2$ now implies \eqref{uniqueness_ineq}.
\end{proof}

\section{Higher-Order Regularity}\label{sec:Reg}
We now prove that the solutions to \eqref{MHD_V_proj} (equivalently \eqref{MHD_V}) enjoy $H^s$ regularity whenever $\bu_0, \B_0\in H^s\cap V$ for $s\geq1$.  We note that the calculations in this section are done formally, but can be made rigorous by proving the results at the Galerkin approximation level, and then passing to the limit in a similar manner to the procedure carried out above (see, e.g., \cite{Larios_Titi_2009}).  By the uniqueness of strong solutions, all strong solutions are regular.

\begin{theorem}\label{thm:reg}
   Consider \eqref{MHD_V_proj} under periodic boundary conditions $\Omega = \nT^3$.  Suppose for $s\geq 1$ that $\bu_0\in H^s\cap V$, $\B_0\in H^s\cap V$.   Then \eqref{MHD_V} (equivalently \eqref{MHD_V_proj}) has a unique strong solution $(\bu,\B)$ with $\bu\in L^\infty([0,T], H^s\cap V)$, $\B\in L^\infty([0,T],H^s\cap V)\cap L^2([0,T],H^{s+1}\cap V)$.
   If we furthermore assume that $\bu_0\in H^{s+1}\cap V$, then we also have $\bu\in L^\infty([0,T], H^{s+1}\cap V)$.
\end{theorem}

\begin{remark}
   It is possible to extend Theorem \ref{thm:reg} to include the case of a specific type of Gevrey regularity, which is analytic in space, as was done for the Euler-Voigt equations in \cite{Larios_Titi_2009}.  For the sake of brevity, we do not pursue such matters here.  For more on Gevrey regularity, see, e.g.,  \cite{Ferrari_Titi_1998, Foias_Manley_Rosa_Temam_2001, Foias_Temam_1989, Levermore_Oliver_1997, Rodino_1993, Kukavica_Vicol_2009, Oliver_Titi_2000,Kalantarov_Levant_Titi_2009,Cao_Rammaha_Titi_1999,Cao_Rammaha_Titi_2000,Ferrari_Titi_1998,Paicu_Vicol_2009} and the references therein.
\end{remark}

\begin{proof}  As indicated above, we only give a formal proof, but the details can be made rigorous by using the Galerkin approximation procedure (see, e.g., \cite{Larios_Titi_2009} for a complete discussion of this method applied to proving higher-order regularity in the context of the Euler-Voigt equations).  Furthermore, we only prove the result in the cases $s=1,2$, as the cases $s>2$ are more complicated notationally, but not conceptually (see, e.g., \cite{Larios_Titi_2009}).  In the case $s=1$, the first statement has already been settled by Theorems \ref{thm:strong} and \ref{thm:unique}.  To prove the second statement, assume that $\bu_0\in\mathcal{D}(A)$ and $\B_0\in V$. Let $\beta$ be a multi-index with $|\beta|=1$.  Applying  $\partial^\beta$ to \eqref{MHD_V_mom}, integrating the result against $\partial^\beta\bu$, and integrating by parts, we have

\begin{align}
\notag
 \frac{1}{2}\frac{d}{dt}\pnt{\alpha^2 \|\partial^\beta \bu\|^2+|\partial^\beta \bu|^2}
 &=
 (\partial^\beta\B\cdot\nabla \B,\partial^\beta \bu)+(\B\cdot\nabla \partial^\beta\B,\partial^\beta \bu)
 \\&\quad \notag
 -((\partial^\beta \bu\cdot\nabla) \bu),\partial^\beta \bu)-(B(u\cdot\nabla) \partial^\beta \bu),\partial^\beta \bu)
 \\\label{MHD_reg_u}&=
 (\partial^\beta\B\cdot\nabla \B,\partial^\beta \bu)+(\B\cdot\nabla \partial^\beta\B,\partial^\beta \bu)
 \\&\quad
 -((\partial^\beta \bu\cdot\nabla) \bu),\partial^\beta \bu).
\notag
\end{align}
Where we have used \eqref{B:alt}. Next, we apply $\partial^\beta$ to \eqref{MHD_V_mag}, integrate the result against $\partial^\beta \B$, and integrate by parts to find
\begin{align}
\notag
 \frac{1}{2}\frac{d}{dt}|\partial^\beta \B|^2 + \mu \|\partial^\beta \B\|^2
 &=
 (\partial^\beta\B\cdot\nabla \bu,\partial^\beta \B)+((\B\cdot\nabla) \partial^\beta \bu,\partial^\beta \B)
 \\&\quad \notag
 -((\partial^\beta \bu\cdot\nabla) ,\B),\partial^\beta \B)-((u\cdot\nabla) \partial^\beta \B),\partial^\beta \B)
 \\\label{MHD_reg_B}&=
 (\partial^\beta\B\cdot\nabla \bu,\partial^\beta \B)+((\B\cdot\nabla) \partial^\beta \bu),\partial^\beta \B)
 \\&\quad \notag
 -((\partial^\beta \bu\cdot\nabla) \B),\partial^\beta \B).
\notag
\end{align}
Since $((\B\cdot\nabla) \partial^\beta\bu,\partial^\beta \B)=-((\B\cdot\nabla) \partial^\beta\B,\partial^\beta \bu)$ by \eqref{B:alt}, we may add \eqref{MHD_reg_u} and \eqref{MHD_reg_B}, to obtain a very important cancellation of the terms that involve the highest order derivatives:
\begin{align}
\notag
&\quad  \frac{1}{2}\frac{d}{dt}
 \pnt{|\partial^\beta \B|^2+\alpha^2 \|\partial^\beta \bu\|^2+|\partial^\beta \bu|^2}
+ \mu \|\partial^\beta \B\|^2
 \\\notag&=
 ((\partial^\beta\B\cdot\nabla) \B,\partial^\beta \bu)
  -((\partial^\beta\cdot\nabla  )u,u,\partial^\beta \bu)
 \\&\quad\notag
 +((\partial^\beta\B\cdot\nabla) \bu,\partial^\beta \B)
 -((\partial^\beta\cdot\nabla)  \bu,\B,\partial^\beta \B)
\\\notag&\leq
C\|\B\|\|\B\|_{H^{3/2}}|\triangle\bu|
 +C\|\bu\|\|\bu\|_{H^{3/2}}|\triangle\bu|
\\\notag&\leq
C\|\B\|^{3/2}|\triangle\B|^{1/2}|\triangle\bu|
+C\|\bu\|^{3/2}|\triangle\bu|^{1/2}|\triangle\bu|
\\\notag&\leq
C\|\B\|^{3/2}|\triangle\B|^{1/2}|\triangle\bu|
+K_\alpha^1|\triangle\bu|^{3/2}
 \\\notag&\leq
 C\|\B\|^{3/2}|\triangle\B|^{1/2}|\triangle\bu|
+K_\alpha^1|\triangle\bu|^{3/2}
\\\notag&\leq
 \frac{\mu}{2}|\triangle\B|^2
 +C\|\B\|^2|\triangle\bu|^{4/3}
+K_\alpha^1|\triangle\bu|^{3/2},
 \end{align}
since $\|\bu\|$ is uniformly bounded on $[0,T]$. Summing over all $\beta$ with $|\beta|=1$, we have
\begin{align}\label{s1_bound}
   &\quad \frac{d}{dt}\pnt{
   \|\B\|^2 + \alpha^2 |\triangle\bu|^2+\|\bu\|^2}
 +\mu |\triangle\B|^2
\leq
 C\|\B\|^2|\triangle\bu|^{4/3}
+K_\alpha^1|\triangle\bu|^{3/2}.
\end{align}
Letting $y=1+\|\B\|^2 + \alpha^2 |\triangle\bu|^2+\|\bu\|^2$ and dropping the term $\mu |\triangle\B|^2$ for a moment, we have an equation of the form
\[
\dot{y}\leq K(t)y^{3/4},
\]
where $K(t) = C \cdot K_\alpha^1(1+\|\B(t)\|^2)$.
Gr\"onwall's inequality gives
\begin{align}
\notag&\quad
\|\B(t)\|^2 + \alpha^2 |\triangle\bu(t)|^2+\|\bu(t)\|^2
\\&\leq K_\alpha^5:=\label{K5_def}
\|\B_0\|^2 + \alpha^2 \|\triangle\bu_0\|^2+\|\bu_0\|^2
+K_\alpha^1\pnt{\int_0^T(1+\|\B(s)\|^2)\,ds}^4.
\end{align}
Since $\int_0^{T}\|\B(s)\|^2\,ds < \infty$ by Theorem \eqref{thm:strong}, we see that $\bu\in L^\infty([0,T],\mathcal{D}(A))$, thanks to \eqref{K5_def} and the norm equivalence \eqref{elliptic_reg}.  This in turn implies that the right-hand side of \eqref{s1_bound} is finite on $[0,T]$.  Integrating \eqref{s1_bound} on $[0,T]$, we find that $\B\in L^2([0,T],\mathcal{D}(A))$.  Thus, we have formally established the theorem in the case $s=1$.

 Let us now take $s=2$.  To begin, we formally take the inner product of \eqref{MHD_V_mom} with $\triangle^2\bu$ (recalling that, in the periodic case, $A=-\triangle$) and integrate by parts several times to obtain
\begin{align}
&\quad \notag
 \frac{1}{2}\frac{d}{dt}\pnt{\alpha^2 \|\triangle \bu\|^2+|\triangle \bu|^2}
 \\&=\notag
 (\triangle\B\cdot\nabla \B,\triangle \bu)
 +2(\nabla\B\cdot\nabla \nabla\B,\triangle \bu)
 +(\B\cdot\nabla \triangle\B,\triangle \bu)
 \\&\quad \notag
 -(\triangle \bu\cdot\nabla\bu,\triangle \bu)
 -2 (\nabla\bu\cdot\nabla \nabla\bu,\triangle \bu)
 -(\bu\cdot\nabla\triangle \bu,\triangle \bu)
 \\\label{MHD_reg2_u}&=
 (\triangle\B\cdot\nabla \B,\triangle \bu)
 +2(\nabla\B\cdot\nabla \nabla\B,\triangle \bu)
 +(\B\cdot\nabla \triangle \B,\triangle\bu)
 \\&\quad \notag
 -(\triangle \bu\cdot\nabla\bu,\triangle \bu)
 -2 (\nabla\bu\cdot\nabla \nabla\bu,\triangle \bu).
\notag
\end{align}
where we have used \eqref{B:alt}. Next, we  take the inner product of   \eqref{MHD_V_mom} with $\triangle^2\B$ and integrate by parts several times to obtain
\begin{align}
&\quad \notag
 \frac{1}{2}\frac{d}{dt}|\triangle \B|^2
 +\mu \|\triangle \B\|^2
 \\&=\notag
 (\triangle\B\cdot\nabla \bu,\triangle \B)
 +2(\nabla\B\cdot\nabla \nabla\bu,\triangle \B)
 +(\B\cdot\nabla \triangle\bu,\triangle \B)
 \\&\quad \notag
 -(\triangle \bu\cdot\nabla\B,\triangle \B)
 -2 (\nabla\bu\cdot\nabla \nabla\B,\triangle \B)
 -(\bu\cdot\nabla\triangle \B,\triangle \B)
 \\\label{MHD_reg2_B}&=
  (\triangle\B\cdot\nabla \bu,\triangle \B)
 +2(\nabla\B\cdot\nabla \nabla\bu,\triangle \B)
 -(\B\cdot\nabla \triangle\B,\triangle \bu)
 \\&\quad \notag
 -(\triangle \bu\cdot\nabla\B,\triangle \B)
 - 2(\nabla\bu\cdot\nabla \nabla\B,\triangle \B).
\notag
\end{align}
where we have again used \eqref{B:alt}.

To prove the first statement of the theorem with $s=2$, we take $\bu_0,\B_0\in \mathcal{D}(A)$.  We already have $\bu\in L^\infty([0,T],\mathcal{D}(A))$ by the case $s=1$, so it remains to show $\B\in L^\infty([0,T],\mathcal{D}(A))\cap L^2([0,T],H^3\cap V)$. Estimating the right-hand side of \eqref{MHD_reg2_B}, we find
 \begin{align}
 \frac{1}{2}\frac{d}{dt}|\triangle \B|^2 +\mu \|\triangle \B\|^2
 &\leq
 C(|\triangle\B|\|\triangle\B\||\triangle\bu|+6|\triangle\bu||\triangle\B|^{3/2}\|\triangle\B\|^{1/2})
 \\&\leq \notag
 (K^5_\alpha)^{1/2}(|\triangle\B|\|\triangle\B\|+|\triangle\B|^{3/2}\|\triangle\B\|^{1/2})
 \\&\leq \notag
  K_{\alpha,\mu}^6|\triangle\B|^2+\frac{\mu}{2}\|\triangle\B\|^2,
\end{align}
 where $K_{\alpha,\mu}^6=K_{\alpha,\mu}^6(K^5_\alpha,\mu)$.  Here, we have used \eqref{B:263}, \eqref{B:623}, and \eqref{B:i22} in the first inequality, \eqref{K5_def} for the second inequality, and Young's inequality for the last inequality.  Rearranging and using Gr\"onwall's inequality, we find,
 \begin{align}\label{B_LiH2_L2H3}
    |\triangle \B(t)|^2 + \int_0^{t}\|\triangle\B(s)\|^2e^{2K_{\alpha,\mu}^6(t-s)}\,ds \leq |\triangle \B_0|^2 e^{2K_{\alpha,\mu}^6t}.
 \end{align}
Thus, by elliptic regularity, we have $\B\in L^\infty([0,T],\mathcal{D}(A))\cap L^2([0,T],H^3\cap V)$.

We now prove the second statement of the theorem with $s=2$.  This statement can be proven independently of the first statement, so we do not rely on \eqref{B_LiH2_L2H3} in the calculations below.  Let $\bu_0\in H^3\cap V$ and $\B_0\in \mathcal{D}(A)$.  Adding \eqref{MHD_reg2_u} and \eqref{MHD_reg2_B} and again noting the important cancellation of higher-order derivatives, we obtain,
\begin{align}
\notag
&\quad \frac{1}{2}\frac{d}{dt}\pnt{|\triangle \B|^2
 +\alpha^2 \|\triangle \bu\|^2+|\triangle \bu|^2}
 + \mu \|\triangle \B\|^2
 \\&=\notag
  (\triangle\B\cdot\nabla \bu,\triangle \B)
 +2(\nabla\B\cdot\nabla \nabla\bu,\triangle \B)
 -(\triangle \bu\cdot\nabla\B,\triangle \B)
 - 2(\nabla\bu\cdot\nabla \nabla\B,\triangle \B)
 \\&\quad \notag
 +(\triangle\B\cdot\nabla \B,\triangle \bu)
 +2(\nabla\B\cdot\nabla \nabla\B,\triangle \bu)
 -(\triangle \bu\cdot\nabla\bu,\triangle \bu)
 - 2(\nabla\bu\cdot\nabla \nabla\bu,\triangle \bu)
 \\\notag&\leq
C(3|\triangle\B|^{3/2}\|\triangle\B\|^{1/2}|\triangle\bu|+3|\triangle\bu|^{5/2}\|\triangle\bu\|^{1/2}+6\|\B\||\triangle\bu|^{1/2}\|\triangle\bu\|^{1/2}\|\triangle\B\|)
 \\\notag&\leq
C\cdot (1+K_\alpha^5)^2(|\triangle\B|^{3/2}\|\triangle\B\|^{1/2}+\|\triangle\bu\|^{1/2}+\|\triangle\bu\|^{1/2}\|\triangle\B\|)
 \\\notag&\leq
C\cdot (1+K_\alpha^5)^2
(|\triangle\B|^{3/2}\|\triangle\B\|^{1/2}
+\|\triangle\bu\|^{1/2}
+\|\triangle\bu\|^{1/2}\|\triangle\B\|)
 \\\notag&\leq
C_\mu\cdot (1+K_\alpha^5)^2
(|\triangle\B|^2
+\|\triangle\bu\|^{1/2}
+\|\triangle\bu\|)+\frac{\mu}{2}\|\triangle\B\|^2
.  \end{align}
The first inequality is due to \eqref{B:est}.  The second inequality is due to \eqref{K5_def}.  The third inequality follows from the norm equivalence \eqref{elliptic_reg}, and the last inequality is due to Young's inequality.
 This leads to
\begin{align}
\notag
&\quad
\frac{d}{dt}\pnt{|\triangle \B|^2
 +\alpha^2 \|\triangle \bu\|^2+|\triangle \bu|^2}
 + \mu \|\triangle \B\|^2
 \\&\leq\notag
 K_{\alpha,\mu}^7
(1+|\triangle \B|^2 +\alpha^2 \|\triangle \bu\|^2+|\triangle \bu|^2)
,  \end{align}
where $K_{\alpha,\mu}^7=C_{\alpha,\mu}\cdot(1+K_\alpha^3)^2$.  Employing Gr\"onwall's inequality, we have
\begin{align}
\notag
&\quad
|\triangle \B(t)|^2 +\alpha^2 \|\triangle \bu(t)\|^2+|\triangle \bu(t)|^2
+\mu\int_0^t \|\triangle \B\|^2e^{K_{\alpha,\mu}^7(t-s)}\,ds
 \\&\leq\notag
(1+|\triangle \B_0|^2 +\alpha^2 \|\triangle \bu_0\|^2+|\triangle \bu_0|^2)e^{K_{\alpha,\mu}^7t}
\end{align}
Due to the equivalence of norms given in \eqref{elliptic_reg}, we see from the above inequality that $\B\in L^\infty((0,T),\mathcal{D}(A))\cap L^2((0,T),H^3\cap V)$ (which was obtained independently in \eqref{B_LiH2_L2H3}), and $\bu\in L^\infty((0,T),H^3\cap V)$.  Thus, we have formally proven the theorem in the case $s=2$.  A similar argument can be carried out inductively for the cases $s>2$.  See, e.g., \cite{Larios_Titi_2009} for a rigorous, detailed discussion of this type of argument.
 \end{proof}

\begin{remark}\label{r:visc_forc}
   Note that one may add a diffusion term $\nu\triangle\bu$ to the right-hand side of  \eqref{MHD_V_mom} or a suitable smooth forcing term to one or both equations \eqref{MHD_V_mom} and \eqref{MHD_V_mag} and recover the same existence, uniqueness, and regularity results obtained above using only slightly modified techniques.  Furthermore, doing so allows one to study the attractor of such a system, which has been done in the case of the Navier-Stokes-Voigt equations in \cite{Kalantarov_Titi_2009, Kalantarov_Levant_Titi_2009}.  We will study these ideas in a forthcoming paper.
\end{remark}

\section{Convergence and a Criterion for Blow-up}\label{sec:conv_blow_up}
In this section, we prove that solutions of the regularized system \eqref{MHD_V} converge to solutions of the non-regularized system (that is, the equations obtained by formally setting $\alpha=0$ in \eqref{MHD_V}).  We first state a short-time existence and uniqueness theorem which was proven in a slightly more general context in \cite{Alekseev_1982} (see also \cite{Schmidt_1988,Secchi_1993}).
\begin{theorem}\label{thm_short_existence_MHD}
   Given initial data $\bu_0, \B_0 \in H^3\cap V$, there exists a time $T>0$ and a unique solution $(\bu,\B)$ to the system \eqref{MHD}, with $\nu=0$, $\mu>0$, such that $\bu,\B\in C([0,T],H^3\cap V)$.
\end{theorem}
Theorem \ref{thm_short_existence_MHD} can be proven similarly to the case of the 3D Euler equations (see, e.g., \cite{Marchioro_Pulvirenti_1994,Majda_Bertozzi_2002}).  For example the existence of solutions can be proven by considering the finite-dimensional Galerkin approximations to \eqref{MHD} based on the eigenfunctions of the stokes operator, and showing that the sequence of solutions as the dimension increases is a Cauchy sequence, and that it converges to a solution of \eqref{MHD} in an appropriate sense.  Theorem \ref{thm_short_existence_MHD} allows us to prove the following convergence theorem.


\begin{theorem}[Convergence as $\alpha\maps0$]\label{t:Convergence} Let $\bu_0, \B_0, \bu_0^\alpha,\B_0^\alpha \in H^3\cap V$.  Let $(\bu,\B)$ be the solution to \eqref{MHD} with $\nu=0$, $\mu>0$, and with initial data $(\bu_0,\B_0)$.  Let $(\bu^\alpha,\B^\alpha)$ be the solution to \eqref{MHD_V} with initial data $(\bu_0,\B_0)$ .  Choose an arbitrary  $T\in (0,T_{\text{max}})$, where $T_{\text{max}}\in(0,\infty]$ is the maximal time for which a solution $(\bu,\B)$ exists and is unique.  Suppose that $\bu_0^\alpha\maps\bu_0$ in $V$ and  $\B_0^\alpha\maps\B_0$ in $H$.  Then $\bu^\alpha\maps\bu$ in $L^\infty([0,T],V)$ and  $\B^\alpha\maps\B$ in $L^\infty([0,T],H)\cap L^2([0,T],V)$, as $\alpha\maps0$.
\end{theorem}
\begin{proof}
Under the hypotheses on the initial conditions, Theorem \ref{thm_short_existence_MHD} holds, and thus there exists a time $T>0$ and a unique $(\bu,\B) \in C([0,T],H^3({\nT^2})\cap V)\times C([0,T],H^3({\nT^2})\cap V)$ solving \eqref{MHD} (with $\nu=0$, $\mu>0$, (in particular, it holds that $T_{\text{max}}>0$).  Thanks to Theorem \ref{thm:strong}, we know that there also exists a unique solution to the problem \eqref{MHD_V}, namely  $(\bu^\alpha,\B^\alpha)\in C([0,T], V)\times ( L^2((0,T),\mathcal{D}(A))\cap C([0,T],V))$.  Subtracting the equations corresponding to $(\bu,\B)$ and $(\bu^\alpha,\B^\alpha)$, and recalling the fact that $A=-\triangle$ due to the periodic boundary conditions, we find
   \begin{subequations}\label{MHD_subtract}
\begin{align}
&\quad\label{MHD_subtract_mo}
-\alpha^2 \frac{d}{dt}\triangle\bu^\alpha+\frac{d}{dt}(\bu^\alpha-\bu)
\\&=\nonumber
B (\B^\alpha-\B,\B) +B (\B^\alpha, \B^\alpha-\B)
-B (\bu^\alpha-\bu,\bu) -B (\bu^\alpha, \bu^\alpha-\bu) ,
\\&\quad\label{MHD_subtract_mag}
\frac{d}{dt}(\B^\alpha-\B) -\mu \triangle(\B^\alpha-\B)
\\&=\nonumber
B (\B^\alpha-\B,\bu) +B (\B^\alpha, \bu^\alpha-\bu)
-B (\bu^\alpha-\bu,\B) -B (\bu^\alpha, \B^\alpha-\B),
\end{align}
\end{subequations}
We now take the inner product of \eqref{MHD_subtract_mo} with $\bu^\alpha-\bu$ and  of \eqref{MHD_subtract_mag} with $\B^\alpha-\B$, and add the results. Using \eqref{B:alt}  and rearranging the terms, we find
\begin{align}
&\quad\label{conv_diff_est} 
 \frac{1}{2}\frac{d}{dt}\pnt{\alpha^2\|\bu-\bu^\alpha\|^2+|\bu-\bu^\alpha|^2
 +|\B-\B^\alpha|^2} + \mu\|\B^\alpha-\B\|^2
 \\&=\notag 
-(B (\bu^\alpha-\bu,\bu),\bu^\alpha-\bu)
-(B(\B^\alpha-\B ,\B),\bu^\alpha-\bu)
\\&\quad\notag-
(B (\bu^\alpha-\bu,\B),\B^\alpha-\B)
-(B(\B^\alpha-\B ,\bu),\B^\alpha-\B)
+\alpha^2\pair{\triangle\bu_t}{\bu^\alpha-\bu}
 \\&\leq\notag 
|\bu-\bu^\alpha|^2\|\nabla\bu\|_{L^\infty}
+2|\B-\B^\alpha|\|\nabla\B\|_{L^\infty}|\bu^\alpha-\bu|
+|\B-\B^\alpha|^2\|\nabla\bu\|_{L^\infty}
\\&\quad\notag
+\alpha^2\pair{\triangle\bu_t}{\bu^\alpha-\bu}
\\&\leq \notag 
K_1(|\bu-\bu^\alpha|^2
+|\B-\B^\alpha|^2)
+\alpha^2\pair{\triangle\bu_t}{\bu^\alpha-\bu},
\end{align}
where $K_1:=C\sup_{[0,T]}\max\set{\|\nabla\bu\|_{L^\infty},\|\nabla\B\|_{L^\infty}}<\infty$, and where we have used Young's inequality.  It remains to estimate the last term on the right-hand side. Using the fact that $\bu$ satisfies \eqref{MHD} (with $\nu=0$), we have
\begin{align}
   &\quad \label{conv_est_H3}
   \alpha^2\pair{ \triangle\bu_t}{\bu^\alpha-\bu}
\\&=\notag
  \alpha^2\pair{ \triangle[\B\cdot\nabla\B-\bu\cdot\nabla\bu]}{\bu^\alpha-\bu}
  \\&=\notag
  \alpha^2\pair{ \triangle\B\cdot\nabla \B+2(\nabla\B\cdot\nabla) \nabla\B+\B\cdot\nabla \triangle\B }{\bu^\alpha-\bu}
\\&\quad-\notag
\alpha^2\pair{ \triangle\bu\cdot\nabla \bu+2(\nabla\bu\cdot\nabla) \nabla\bu+\bu\cdot\nabla \triangle\bu }{\bu^\alpha-\bu}
 \\&\leq\notag
  C\alpha^2\|\B\|^{1/2}\|\B\|_{H^2}^{3/2}\|\bu-\bu^\alpha\|
 +C\alpha^2\|\B\|_{H^2}^2\|\bu-\bu^\alpha\|
 \\&\quad+\notag
 C\alpha^2\|\bu\|^{1/2}\|\bu\|_{H^2}^{3/2}\|\bu-\bu^\alpha\|
 +C\alpha^2\|\bu\|_{H^2}^2\|\bu-\bu^\alpha\|
 \\&\leq\notag
K_2\alpha^2+\alpha^2\|\bu-\bu^\alpha\|^2
\end{align}
where $K_2:=C\sup_{[0,T]}\max\set{\|\B\|^{1/2}\|\B\|_{H^2}^{3/2},\alpha^2\|\B\|_{H^2}^2,\|\bu\|^{1/2}\|\bu\|_{H^2}^{3/2},\|\bu\|_{H^2}^2}<\infty$.
Combining \eqref{conv_diff_est} with \eqref{conv_est_H3} yields
\begin{align}
&\quad\notag
\frac{d}{dt}\pnt{\alpha^2\|\bu-\bu^\alpha\|^2+|\bu-\bu^\alpha|^2
 +|\B-\B^\alpha|^2} + 2\mu\|\B^\alpha-\B\|^2
 \\&\leq\notag
 K_3
 (
 \alpha^2\|\bu-\bu^\alpha\|^2
 +|\bu-\bu^\alpha|^2
+|\B-\B^\alpha|^2
)
 +\alpha^2K_2,
 \end{align}
where $K_3=C\max\set{1,K_2}$.  Using Gr\"onwall's inequality, we have
\begin{align}
  &\quad\notag
  \alpha^2\|\bu(t)-\bu^\alpha(t)\|^2+|\bu(t)-\bu^\alpha(t)|^2+|\B(t)-\B^\alpha(t)|^2
  \\&\quad+\notag
  2\mu\int_0^t\|\B^\alpha(s)-\B(s)\|^2e^{K_3(t-s)}\,ds
 \\&\leq\notag
 C\pnt{\alpha^2\|\bu_0-\bu^\alpha_0\|^2+|\bu_0-\bu^\alpha_0|^2
 +|\B_0-\B^\alpha_0|^2} + \alpha^2\frac{K_2}{K_3}(1-e^{K_3t}).
\end{align}
Thus, if $\bu^\alpha_0\maps \bu_0$ in $V$ and $\B_0^\alpha\maps \B_0$ in $H$ as $\alpha\maps0$ (in particular, if $\bu^\alpha_0= \bu_0$ and $\B_0^\alpha= \B_0$ for all $\alpha>0$), then $\bu^\alpha\maps \bu$ in $L^\infty([0,T],V)$ and $\B^\alpha\maps \B$ in $L^\infty([0,T],H)\cap L^2([0,T],V)$ as $\alpha\maps0$.
\end{proof}

\begin{theorem}[Blow-up criterion]\label{t:blow_up_criterion}
With the same notation and assumptions of Theorem \ref{t:Convergence}, and taking $\bu_0^\alpha=\bu_0$ and $\B_0^\alpha=\B_0$ for all $\alpha>0$, suppose that for some $T_*<\infty$, we have
\begin{align}\label{blow_up_ineq}
   \sup_{t\in[0,T_*)}\limsup_{\alpha\maps0}\alpha^2\|\bu^\alpha(t)\|^2>0.
\end{align}
Then the solutions to \eqref{MHD} with $\nu=0$, $\mu\geq0$ become singular in the time interval $[0,T_*)$.
\end{theorem}
\begin{proof}
   Suppose for a contradiction that $\bu$ and $\B$ are bounded in $L^\infty([0,T],H^3({\nT^2})\cap V)$, but that \eqref{blow_up_ineq} holds.   Taking the inner product of \eqref{MHD_V_mom} with $\bu^\alpha$ and the inner product of \eqref{MHD_V_mag} with $\B^\alpha$ and adding the results, we find after using \eqref{B:alt},
   \begin{align*}
      &\alpha^2\|\bu^\alpha(t)\|^2
      +|\bu^\alpha(t)|^2
       +|\B^\alpha(t)|^2
       +\mu\int_0^t\|\B^\alpha(s)\|^2\,ds
      =
        \alpha^2\|\bu_0\|^2
        +|\bu_0|^2
      +|\B_0|^2.
   \end{align*}
Taking the $\limsup$ as $\alpha\maps 0^+$ gives
\begin{align}\label{blow_up_limsup}
      \limsup_{\alpha\maps0^+}\alpha^2\|\bu(t)\|^2
      +|\bu(t)|^2
      +|\B(t)|^2
       +\mu\int_0^t\|\B(s)\|^2\,ds
     =
        |\bu_0|^2
      +|\B_0|^2.
   \end{align}

However, thanks to Theorem \ref{thm_short_existence_MHD}, we have $\bu,\B\in C([0,T],H^3\cap V)$.  Using Lions' Lemma (see, e.g., \cite[p. 176]{Temam_2001_Th_Num} or \cite[Corollary 7.3]{Robinson_2001}) and \eqref{B:alt} it is straight-forward to prove from \eqref{MHD} (with $\nu=0$, $\mu>0$), that on $[0,T]$,
\begin{align*}
      |\bu(t)|^2
      +|\B(t)|^2
       +\mu\int_0^t\|\B(s)\|^2\,ds
     =
        |\bu_0|^2
      +|\B_0|^2.
   \end{align*}
   so that \eqref{blow_up_limsup} contradicts \eqref{blow_up_ineq}.
\end{proof}

\begin{remark}
   We note that that in the case $\mu,\nu>0$, theorems similar to Theorems \ref{thm_short_existence_MHD}, \ref{t:Convergence}, and \ref{t:blow_up_criterion} can be proven with weaker assumptions on the initial data.  In the case of Theorem \ref{thm_short_existence_MHD}, one can use, for example, Galerkin methods, and ideas similar to those of in the theory of the Navier-Stokes equations (see, e.g., \cite{Constantin_Foias_1988, Temam_2001_Th_Num}).  The analogous of Theorems \ref{t:Convergence} and \ref{t:blow_up_criterion} can be proven using nearly identical methods to those employed above.
\end{remark}

\section*{Acknowledgements}
The authors are thankful for the warm hospitality
of  the Institute for Mathematics and its
Applications (IMA), University of Minnesota, where
part of this work was completed. This work was
supported in part by the NSF grants
no.~DMS-0708832, DMS-1009950. E.S.T. also
acknowledges the kind hospitality of the Freie
Universit\"at - Berlin, and the support of the
Alexander von Humboldt Stiftung/Foundation and the
Minerva Stiftung/Foundation.

\begin{scriptsize}

\providecommand{\bysame}{\leavevmode\hbox to3em{\hrulefill}\thinspace}
\providecommand{\MR}{\relax\ifhmode\unskip\space\fi MR }
\providecommand{\MRhref}[2]{%
  \href{http://www.ams.org/mathscinet-getitem?mr=#1}{#2}
}
\providecommand{\href}[2]{#2}

\end{scriptsize}

\end{document}